\documentclass[12pt]{amsart}
\usepackage[mathcal]{eucal}
\usepackage{amsmath}
\usepackage{amsthm}
\usepackage{color}
\usepackage{amscd}
\usepackage{amssymb}
\usepackage{tikz-cd}

\DeclareMathOperator{\sind}{Sind}
\DeclareMathOperator{\ch}{ch}
\DeclareMathOperator{\chr}{chr}
\DeclareMathOperator{\Sch}{Sch}
\DeclareMathOperator{\TSch}{TSch}

\DeclareMathOperator{\tr}{tr}
\DeclareMathOperator{\Tr}{Tr}

\DeclareMathOperator{\Td}{Td}
\DeclareMathOperator{\Ker}{Ker}
\DeclareMathOperator{\Ind}{Ind}

\DeclareMathOperator{\Tch}{Tch}

\DeclareMathOperator{\ind}{ind}
\DeclareMathOperator{\ord}{ord}
\DeclareMathOperator{\Res}{Res}

\DeclareMathOperator{\cone}{Cone}

\DeclareMathOperator{\sgn}{sgn}

\DeclareMathOperator{\Id}{Id}

\DeclareMathOperator{\htr}{\widehat{\Tr}}

\newcommand{\cs}{\mathcal{S}}

\newtheorem{thm}{Theorem}[section]
\newtheorem{cor}[thm]{Corollary}
\newtheorem{prop}[thm]{Proposition}
\newtheorem{lemma}[thm]{Lemma}

\newtheorem{rem}[thm]{Remark}

\theoremstyle{remark}
\newtheorem{remark}[thm]{Remark}

\newtheorem{defn}{Definition}

\newcommand{\cA}{\mathcal{A}}
\newcommand{\cB}{\mathcal{B}}

\newcommand{\cJ}{\mathcal{J}}

\newcommand{\cL}{\mathcal{L}}

\newcommand{\cS}{\mathcal{S}}

\newcommand{\slch}{\,/\!\!\!\!\operatorname{ch}}

\def \C{{\mathbb C}}

\def \Z{{\mathbb Z}}

\begin{document}

\title{Index pairing with Alexander-Spanier cocycles}

\author{Alexander Gorokhovsky}
\address{Department of Mathematics,
University of Colorado, Boulder, CO 80309-0395,USA
 }

\email{alexander.gorokhovsky@colorado.edu}

\thanks{The work of A.G. was partially supported by the National Science Foundation
 award DMS-0900968}

\author{Henri Moscovici}
\address{Department of Mathematics,
The Ohio State University,
Columbus, OH 43210,
USA}
\email{moscovici.1@osu.edu}

\thanks{The work of H.M. was partially supported by the National Science Foundation
 award DMS-1600541}
\dedicatory{To ALAIN CONNES, with admiration and deep appreciation}

\begin{abstract}

We give a uniform construction of the higher indices of elliptic operators
associated to Alexander--Spanier cocycles of either parity in terms of a pairing \`a la
Connes between the $K$-theory and the cyclic cohomology of
the algebra of complete symbols of  pseudodifferential operators, implemented
by means of a relative form of the Chern character in cyclic homology.
While the formula for the lowest index of
an elliptic operator $D$ on a closed manifold $M$ (which coincides with its Fredholm index)
reproduces the Atiyah-Singer index theorem,
our formula for the highest index of $D$ (associated to a volume cocycle)
yields an extension to arbitrary manifolds of any dimension
of the Helton--Howe formula for the trace of multicommutators of classical Toeplitz operators on odd-dimensional spheres. In fact, the
totality of higher analytic indices for an elliptic operator $D$ amount to a representation of the
Connes-Chern character of the $K$-homology cycle determined by $D$ in terms of expressions
which extrapolate the Helton--Howe formula below the dimension of $M$.
\end{abstract}

\maketitle

\section{Introduction}
The intuition that the Alexander--Spanier version of cohomology was an yet untapped
resource for index theory has been one of the surprising insights that made the original
list of topics in Alain Connes'  master plan for noncommutative geometry
~\cite[Introduction]{Connes_ncdg}. It has materialized in~\cite{CM1990}, where
higher indices of elliptic operators associated to Alexander--Spanier cocycles
were used to prove
the Novikov conjecture for manifolds with word-hyperbolic fundamental group.

For an elliptic operator $D$
between two vector bundles over a closed manifold $M$, the higher analytic index
$\Ind_{\phi} D \in \C$ corresponding to an even-dimensional Alexander--Spanier
cocycle $\phi$ on $M$ was constructed essentially
by folding the Alexander--Spanier cohomology of the manifold into cyclic cohomology
for pseudodifferential operators, by exploiting the fact that the parametrices
of $D$ can be localized at will near the diagonal. The resulting number $\Ind_{\phi} D $
depends only on the cohomology class $[\phi] \in H^{ev}(M, \C)$; in particular
$\Ind_{1} D$  coincides with the Fredholm index of $D$. After showing that these higher
indices admit cohomological expressions akin to the Atiyah-Singer
index formula~\cite{AS3}, Connes and Moscovici proved a generalization of the $\Gamma$-index
theorem (cf. Atiyah~\cite{At_Gamma} and Singer~\cite{Sin_Gamma}), which was
instrumental in their
proof~\cite{CM1990} of the homotopy invariance of higher signatures.

In a different role, the pairing between Alexander--Spanier cohomology and the signature
operator was used to produce local expressions for the rational Pontryagin classes
of topological manifolds (quasiconformal in~\cite{CoSuTel}
and Lipschitz in~\cite{MW1994}), and also for the Goresky-MacPherson
$\mathcal{L}$-class of Witt spaces~\cite{MW1995, MW1996}.

In the aforementioned applications the odd case of the pairing, associating
higher analytic indices of selfadjoint elliptic operators to
odd-dimensional Alexander--Spanier cocycles, was handled in an indirect manner,
essentially by reducing it via Bott suspension to the even case.
The lack of a natural uniform definition irrespective of parity
remained though a challenging conundrum, at least in the minds of the present authors.
It is the purpose of this paper to provide a unified construction
for higher analytic indices of either parity, based on recasting their definition
within the conceptual framework of Connes' pairing
between the $K$-theory and the cyclic cohomology of an algebra. Instead of the algebra of
functions, the appropriate algebra for the task at hand
 is the algebra of complete symbols of pseudodifferential
operators, over which the pairing is implemented by means of
a relative form of the Chern character in cyclic homology.

Perhaps the quickest way to illustrate the flavor of the present picture for the
higher index pairing  (described in \S \ref{HAI} below)
is to mention its extreme cases. While $\Ind_1 D$ is just
the Fredholm index of the elliptic operator $D$, the higher analytic
index of $D$ associated to a top-dimensional Alexander--Spanier
(volume) cocycle on a manifold $M$ has an expression reminiscent of
the Helton--Howe formula~\cite[\S 7]{HH1975} for the trace of the
top multicommutator of Toeplitz operators on an odd-dimensional sphere.
As a matter of fact our formula represents a twofold extension
of the Helton--Howe formula, to arbitrary manifolds and to any dimension,
regardless of parity.

It was already noticed by Connes (cf.~\cite[Part I, \S7]{Connes_ncdg}) that
 in the commutative case
 the total antisymmetrization of his Chern character in K-homology yields
the Helton--Howe fundamental trace form.
The cohomological formulas (proved in \S \ref{CoFo} below) for
the higher analytic indices of an elliptic operator $D$ on a closed manifold $M$ show that,
conversely, the full Connes-Chern character of the K-homology cycle determined by $D$
can be recovered by extrapolating the Helton--Howe formula below the dimension
of $M$.

\section{Higher analytic indices} \label{HAI}

\subsection{Higher Alexander--Spanier traces} \label{AS2CC}
Let $M$ be a closed manifold.
The (smooth) Alexander--Spanier complex of  $M$ is the quotient
complex \, ${\bf C}^\bullet(M) =\{C^\bullet(M)/C_0^\bullet(M), \delta \}$, \, where
$C^q(M) =C^\infty(M^{q+1}) $,
 $C_0^q(M)$ consists of those $\phi \in C^q(M)$ which vanish in a neighborhood
of the iterated diagonal $\Delta_{q+1}M = \{(x, \ldots, x) \in M^{q+1}\}$, and
 \begin{equation*}
 \delta \phi(x_0, x_1, \ldots, x_{q+1})  =
\sum (-1)^i \phi(x_0, \ldots, \widehat{x_i}, \ldots, x_{q+1}) .
\end{equation*}
Its cohomology groups $H^\bullet_{AS}(M)$ yield the usual cohomology of $M$,
alternatively computed from the \v{C}ech, the simplicial or the de Rham complex.
The same cohomology is obtained from several variants of the Alexander--Spanier complex.
We single out one of these variations which is particularly suited to the purposes of
this paper. The specific complex consists of {\em decomposable} Alexander--Spanier cochains,
i.e. finite sums of the form
\begin{align*}
\phi \, =  \sum_\alpha f^\alpha_0 \otimes f^\alpha_1 \otimes \ldots \otimes f^\alpha_q ,
 \qquad f^\alpha_i \in C^\infty(M) ,
\end{align*}
 which in addition are totally antisymmetric
 \begin{align*}
  \phi (x_{\nu(0)}, x_{\nu(1)}, \ldots, x_{\nu(q)})\,  =
  \sgn(\nu) \phi(x_0, x_1, \ldots, x_q) ,
  \qquad \forall \, \nu \in S_{q+1} .
\end{align*}
The decomposable and totally antisymmetric cochains give rise to
a quasi-isomorphic subcomplex
${\bf C}^\bullet_\wedge(M) = \{C_\wedge^\bullet(M)/C_{\wedge, 0}^\bullet(M), \delta \}$
of the Alexander--Spanier complex ${\bf C}^\bullet(M)$.

\medskip

On the other hand we consider the algebra of classical pseudodifferential operators
$\Psi (M)$ on $M$, and the exact sequence
\begin{equation} \label{PsiX}
0 \rightarrow \Psi^{-\infty}(M) \rightarrow \Psi (M)  \overset{\sigma}{\longrightarrow}
 \cS(M)   \rightarrow 0 ,
\end{equation}
where $\Psi^{-\infty}(M) $ is the ideal of smoothing operators
 and  $  \cs(M)$ is the quotient algebra
of complete symbols; $\sigma$ denotes the complete symbol map.

To any cochain $\phi = \sum_i f^i_0\otimes f^i_1 \otimes \ldots \otimes f^i_k \in C_\wedge^k(M)$
we associate the multilinear form $\Tr_\phi$ on $\Psi^{-\infty}(M)$, defined by
 \begin{align} \label{tr-phi}
\Tr_\phi (A_0, A_1, \ldots, A_k) =  \sum_i \Tr \left(A_0 f^i_0 A_1 f^i_1 \ldots A_k  f^i_k \right) ,
\end{align}
The assignment $ C_\wedge^\bullet(M) \ni \phi \mapsto \Tr_\phi \in CC^\bullet(\Psi^{-\infty}(M))$
 satisfies the coboundary identities
\begin{align} \label{b-phi}
b \Tr_\phi = \Tr_{\delta \phi} , \quad B \Tr_\phi = 0 \, ;
\end{align}
the first one is tautological and the second follows from the total antisymmetry
of $\phi$.

 To extend this assignment to the full algebra of pseudodifferential operators
we use zeta-regularization of the operator trace.
Fix an elliptic operator $R\in \Psi^1(M)$, invertible and positive.
For $A \in \Psi^p(M)$ form
 \begin{align} \label{zeta}
\zeta(s) = \Tr A R^{-s} , \qquad \Re (s) >> 0 .
\end{align}
  $\zeta(s)$ has meromorphic continuation beyond $0$, around which
 \begin{align} \label{zeta0}
\zeta(s) = \frac{1}{s} \Res A +\htr A +O(s).
\end{align}
The Wodzicki residue  $\Res A$ is independent
of the choice of $R$, defines a trace on the algebra $\Psi(M)$
and can be computed from its symbol
$\sigma(A)(x, \xi) \sim \sum_{j = 0}^\infty \sigma_{p-j} (A)(x, \xi)$, with $\sigma_{k} (A)(x, \xi)$
homogeneous of degree $k$, by the formula (cf. \cite{Wod1987})
\begin{align} \label{Wres}
\Res A \, = \, (2 \pi)^{-\dim M}  \int_{S^*M} \sigma_{-\dim M} (A)(x, \xi) |d\xi dx| .
\end{align}
In particular,  $\Res A = 0$ if $A$ is trace class.

On occasion we shall need to involve the logarithm of a positive elliptic
pseudodifferential operator. An appropriate such algebra is that of pseudodifferential operators
with log-polyhomogeneous symbols (cf. \cite{Lesch1999}), which will be denoted here
$\Psi_{\log} (M)$.The residue functional admits a canonical extension to $\Psi_{\log} (M)$,
which for $A \in \Psi_{\log}^{p, q} (M)$ with symbol
\[
\sigma(A)(x, \xi) \sim \sum_{j \ge 0}^\infty \sum_{k = 0}^q \sigma_{p-j, k} (A)(x, \xi) \log^k |\xi|
\]
is given by the local formula \cite[Cor. 4.8]{Lesch1999},
\begin{align} \label{LWres}
\Res_q A \, = \, \frac{(q+1)!}{(2 \pi)^{\dim M}}  \int_{S^*M} \sigma_{-\dim M, q} (A)(x, \xi) |d\xi dx| .
\end{align}

If $\phi = \sum_i f_0^i \otimes f_1^i \otimes \ldots \otimes f_k^i \in C^k (M)$,
 we define the multilinear form $\htr_\phi \in CC^{k}(\Psi(M))$ by
 \begin{align} \label{regtr-phi}
\htr_\phi(A_0, A_1, \ldots, A_k) = \sum_i\htr  \left(A_0 f^i_0 A_1 f^i_1 \ldots A_k  f^i_k \right).
\end{align}
The assignment $\phi \mapsto \htr_\phi$ fails to be a chain map, since
$\htr$ is not a trace. Instead, its Hochschild boundary is given
by a local formula (see e.g. \cite{MN1996}):
\begin{align} \label{bhtr}
b\htr(A, B) \equiv \htr[A, B]= \Res (A [\log R, B]).
\end{align}

\begin{lemma}\label{cyclic}
Let $\phi = \sum_i f_0^i \otimes f_1^i \otimes \ldots \otimes f_k^i \in C_{\wedge,0}^k(M)$ and
let $A_0, \ldots, A_k\in \Psi(M)$.
Then
\begin{align} \label{cyclic1}
\sum\htr  \left(A_0 f_0^i A_1 f_1^i \ldots A_k  f_k^i \right)
&= \sum\htr  \left(f_k^i A_0 f_0^i A_1 f_1^i \ldots A_k   \right) ;\\ \label{cyclic2}
\sum\htr  \left(A_0 f_0^i A_1 f_1^i \ldots A_k  f_k^i \right)
&= \sum\htr  \left( f_0^i A_1 f_1^i \ldots A_kf_k^i A_0   \right) .
\end{align}
\end{lemma}

\begin{proof}
By equation \eqref{bhtr}, one has
 \begin{align*}
\sum\htr  [A_0 f_0^i A_1 f_1^i \ldots A_k,  f_k^i] \,=\,
 \Res \sum_i A_0 f_0^i A_1 f_1^i \ldots A_k [\log R, f_k^i] .
\end{align*}
In the right hand side of \eqref{cyclic1} the $\Res$ functional is applied
to a classical pseudodifferential operator which is a difference of
two operators in $\Psi_{\log}^{\bullet, 1} (M)$, namely
 \begin{align*}
\sum_i A_0 f_0^i A_1 f_1^i \ldots A_k (\log R)  f_k^i  -
 \sum_i  f_k^i A_0 f_0^i A_1 f_1^i \ldots  A_k \log R .
 \end{align*}
Since $\phi$ is locally zero, its jet at the iterated diagonal $\Delta_{k+1}M$
is identically zero.
The symbol multiplication formula then shows that
complete symbol of each of the above two operators vanishes, and so
the identity \eqref{cyclic1} follows from \eqref{Wres}. The identity
\eqref{cyclic2} is checked in a similar fashion.
 \end{proof}

As a consequence of this lemma, we can now reproduce the coboundary
equations \eqref{b-phi} but only for the restriction to locally zero cochains:
\begin{align} \label{b-phi-zero}
b \htr_\phi = \htr_{\delta \phi} , \quad B \htr_\phi = 0 , \quad
 \forall \, \phi \in C_{\wedge,0}^\bullet(M);
\end{align}
the first follows from \eqref{cyclic2}, and the second from
\eqref{cyclic1} combined with the total antisymmetry of $\phi$.

To promote the above assignment to the full complex we define the map
 \begin{align} \label{chi-map}
  \chi_\phi \, = \,  \htr_{\delta \phi} -(b+B)\htr_\phi , \qquad \phi \in C_\wedge^\bullet(M) \, .
\end{align}

\begin{prop} \label{post-cyclic}
The map $ C_\wedge^\bullet(M) \ni \phi \mapsto \chi_\phi \in CC^\bullet(\Psi(M))$
induces a morphism of complexes with a shift of degree
$\chi:  {\bf C}^\bullet_\wedge(M) \to CC^{\bullet+1}(\cs(M))$.
\end{prop}

\begin{proof} Let $\phi \in C_\wedge^q(M)$.
By \eqref{b-phi-zero},  $\chi_\phi=0$ if $\phi $ is locally zero, ensuring that the
map $\phi \mapsto \chi_\phi$ descends to the quotient complex
$ {\bf C}^\bullet_\wedge(M) $. Lemma \ref{cyclic} also implies that
$\chi_\phi(A_0, A_1, \ldots, A_{q+1})=0$ whenever one of
the operators $A_i$'s is trace class, and so $\chi_\phi \in CC^{q+1}(\cs(M))$.
Finally, the map $\chi:  {\bf C}^\bullet_\wedge(M) \to CC^{\bullet +1}(\cs(M))$
is a morphism of complexes since
 \begin{align*}
 (b+B) \chi_\phi \, = \, (b+B) \htr_{\delta \phi} \, = \,- \chi_{\delta \phi}  .
\end{align*}
\end{proof}

We shall mostly work with a variant of this construction, in which the full
pseudodifferential extension \eqref{PsiX} is replaced by
\begin{equation} \label{Psi0X}
0 \rightarrow \cJ^0(M) \rightarrow \Psi^0 (M)  \overset{\sigma}{\longrightarrow}
 \cS^0(M)   \rightarrow 0 ;
 \end{equation}
here $\cJ^0(M) = \cL_1 \cap \Psi^0(M)$ is the ideal of trace class operators
in $\Psi^0 (M)$, and $\cS^0(M)$ is the corresponding quotient algebra
of symbols.  The same reasoning as above yields the chain map
 \begin{align} \label{chi-map-0}
 \begin{split}
  \chi \colon C^\bullet_{\wedge}(M) \rightarrow CC^{\bullet+1}(\cs^0(M)) , \\
  \chi_\phi \, = \,\htr_{\delta \phi}   -(b+B)\htr_\phi , \qquad \phi \in C_\wedge^\bullet(M) \, .
  \end{split}
\end{align}

 The relevance of the map $\chi$ for index theory becomes readily apparent when
 applied to the $0$-dimensional cocycle $\phi \equiv 1$.
 Indeed, if $D \in M_\infty(\Psi^0 (M))$ is elliptic, $u=\sigma(D)$
 and $Q$ is a parametrix for $D$ with $\sigma (Q) = u^{-1}$, then
 \begin{align} \label{chi-map-1}
 \begin{split}
  \chi_1 (u^{-1}, \, &u)  =  -b\htr (Q, D) = \htr (DQ - QD) \\
&=  \Tr \left((\Id - QD) - (\Id - DQ)\right) =  \ind D  \, .
  \end{split}
\end{align}
It will be shown below that the pairing determined by the chain map $\chi$ captures
not just the Fredholm index but also its higher analogues, detecting
the rational $K$-homology class of $D$.

\subsection{Relative Chern character for quotient algebras} \label{RelCh}

To prepare the ground for the discussion of higher analytic indices, we recall
the definition of the relative Chern character in cyclic homology.

We begin by recalling some standard notation
for the cyclic homology (normalized) bicomplex of a unital algebra
$\cA$. Namely, $C_{q}(\cA) :=\cA \otimes (\cA/\C)^{\otimes q}$,
and the boundary operators $b\colon C_{q+1}(\cA) \to C_{q}(\cA) $ and
$B\colon C_{q-1}(\cA) \to C_{q}(\cA) $ are defined by
\begin{align*}
b \varphi (a^0 ,a^1 ,\ldots ,a^{q+1})& :=
           \sum_{j=0}^q  (-1)^j  (a^0 ,\ldots ,a^j a^{j+1}, \ldots , a^{q+1}) \\
&+ (-1)^{q+1} (a^{q+1} a^0 ,a^1 ,\ldots ,a^{q})  ,\\
  B \varphi (a_0 , \ldots ,  a_{q-1})& :=
  \sum_{j=0}^{n-1} (-1)^{(n-1)j} ( 1 , a_j , \ldots , a_{q-1},
  a_0 , \ldots ,  a_{j-1} ) .
\end{align*}
The complex $\left(CC_{q}(\cA)= \bigoplus_{0\leq k \leq q } C_k(\cA) , b +B \right)$ yields
the cyclic homology groups $HC_{q}(\cA)$, while
$\left(CC^{per}_{ev | odd}(\cA)= \prod_{2k | 2k+1 } C_k(\cA), b +B \right)$ yields
the periodic cyclic homology groups $HC^{per}_{ev | odd}(\cA)$.

The datum for relative cyclic homology
consists of a unital  algebra $\cA$ together with a two-sided ideal $\cJ$.
 With $CC_\bullet (\cA)$ and
$CC_\bullet (\cA/\cJ)$ denoting the respective cyclic
homology mixed complexes, and  $\, CC_\bullet(\cA, \cJ) $
standing for the kernel complex, one has the exact sequence of complexes
\begin{align*}
0 \to CC_\bullet(\cA, \cJ) \overset{\iota_*}{\longrightarrow} CC_\bullet (\cA)
\overset{\sigma_*}{\longrightarrow} CC_\bullet (\cA/\cJ) \to 0 ,
\end{align*}
where $\iota \colon \cJ \to \cA$ is the inclusion map and $\sigma \colon \cA \to \cA/\cJ$
is the projection.
The relative cyclic homology is the homology $HC_\bullet (\cA, \cJ)$
of the kernel complex. If $\cJ$ is excisive  (e.g. H-unital) these groups are
are naturally isomorphic to the groups $HC_\bullet (\cJ)$, and in any case
$HC^{per}_{ev | odd}(\cA, \cJ)$ are naturally isomorphic to the groups
$HC^{per}_{ev | odd}(\cJ)$.
The standard homological cone construction
(cf. e. g. \cite[\S 1.1]{LMP2009}) applied to
the projection $\, \sigma_\ast \colon CC_\bullet (\cA) \to CC_\bullet (\cA/\cJ)$ gives an
alternative description to the relative
cyclic homology, and therefore to the cyclic homology of an excisive
two-sided ideal $\cJ$.

For the purposes of this paper we shall need an alternative description of
the cyclic homology of the quotient algebra $\cA/\cJ$, which is obtained
by forming the cone complex of the inclusion map
$i_\ast \colon CC_\bullet(\cA, \cJ) \to CC_\bullet (\cA)$. The resulting
homology groups will be denoted $HC_\bullet (\cA\colon\cJ)$, resp.
$HC^{per}_{ev | odd}(\cA \colon \cJ)$.
Explicitly,
\begin{equation*}
\cone_k \left[CC_\bullet(\cA, \cJ) \to CC_\bullet(\cA)\right]:= CC_{k-1}(\cA, \cJ) \oplus CC_k(\cA)
\end{equation*}
with the differential given by
\begin{equation*}
  (\alpha, \beta) \mapsto \left(-(b+B)\alpha, \iota_*(\alpha) +(b+B) \beta \right)
\end{equation*}
for $(\alpha, \beta) \in  CC_{k+1}(\cA, \cJ) \oplus CC_k(\cA)$.
There is a natural quasi-isomorphism
$\cone \left[CC_\bullet(\cA, \cJ) \to CC_\bullet(\cA)\right] \to CC_\bullet(\cA/\cJ) $,
given by the assignment
\begin{equation*}
(\alpha, \beta) \mapsto \sigma_*(\beta),
\end{equation*}
which induces a canonical isomorphism  $\kappa_\bullet\colon HC_\bullet (\cA\colon\cJ) \to
 HC_\bullet (\cA/\cJ)$ and its periodic version $\kappa_\bullet\colon HC^{per}_{ev | odd}(\cA \colon \cJ) \to
 HC^{per}_{ev | odd}(\cA/\cJ)$.

In the remainder of this subsection $\cA$ is assumed to be a unital Fr\'echet algebra
with the group of invertibles open and continuous inversion;
$\cJ$ is a closed two-sided ideal. The topological $K$-theory of $\cA$ is
related to the periodic cyclic cohomology of $\cA$ via the Chern character
$\ch\colon K_\bullet(\cA) \to HC^{per}_{\bullet}(\cA)$, whose definition
we proceed to recall.

If $E \in M_\infty (\cA)$ is an idempotent representing a class in $K_0(\cA)$, the
corresponding Chern character is given by the $(B,b)$-cocycle in
$CC^{per}_{ev}(\cA)$
\begin{align} \label{ch-ev}
\ch (E) = \tr (E) + \sum_{q \geq 1}(-1)^q \frac{(2q)!}{q!}
 \tr_{\otimes}((E -\frac 12)\otimes E^{\otimes 2q}) ,
\end{align}
and if $U \in GL_\infty (\cA)$ is an invertible then the
representing a class in $K_1(\cA)$, then Chern character of its class
is represented by  the $(B,b)$-cocycle in $CC^{per}_{odd}(\cA)$
\begin{align} \label{ch-odd}
\ch(U) = \sum_{q \geq 0} (-1)^q q! \tr_{\otimes}((U^{-1} \otimes U)^{\otimes q+1}) .
 \end{align}

The same formulas apply, of course, to the quotient algebra, defining
$\ch\colon K_\bullet(\cA/\cJ) \to HC^{per}_{\bullet}(\cA/\cJ)$.
In order to define the relative version of the latter though
we shall also need the transgressed cochains that are canonically
associated to $C^1$-paths of idempotents $E(t)$, resp. invertibles $U(t)$, $t \in [t_0, t_1]$.
They are defined as follows. In the case of a path of idempotents,
  \begin{align}  \label{schs-ev}
\begin{split}
\Tch (E, \dot{E})\colon= & \int_{t_0}^{t_1}  \slch (E (t), \dot{E}(t)) \, dt \, , \qquad  \text{where} \\
 \slch(E , \dot{E})=&\iota ((2E -1)\dot{E}) \left(\ch (E)\right) \qquad \text{with}\\
 \iota (B)(A_0 \otimes \ldots \otimes A_q) &=
 \sum_{j=0}^q (-1)^j A_0 \otimes \ldots \otimes A_j \otimes B \otimes \ldots \otimes A_q ;
\end{split}
 \end{align}
 one has
 \begin{align*}
\frac{d}{dt} \ch E(t) \, = \, (b+B) \slch(E, \dot{E}) ,
 \end{align*}
whence the transgression identity
 \begin{align} \label{trans-ev}
\ch E(t_1) -  \ch E(t_0) \, = \, (b+B) \Tch (E, \dot{E}) .
 \end{align}
 For a path of invertibles
  \begin{align}  \label{schs-odd}
  \begin{split}
\Tch &(U, \dot{U}):= \int_{t_0}^{t_1}  \slch (U (t), \dot{U}(t)) \, dt , \qquad  \text{where} \\
  \slch & (U , \dot{U})= \tr (U^{-1}\dot{U}) + \\
 +\sum_{q=0}^\infty (-1)^{q+1} &q!
 \sum_{j=0}^q \tr_{\otimes} \big( (U^{-1} \otimes U)^{\otimes j+1}\otimes
   U^{-1} \dot{U} \otimes (U^{-1} \otimes U)^{ \otimes q-j} \big) ,
   \end{split}
 \end{align}
 satisfying
 \begin{align*}
\frac{d}{dt} \ch U(t) \, = \, (b+B) \slch(U, \dot{U}) ,
 \end{align*}
whence the transgression identity
 \begin{align} \label{trans-odd}
\ch U(t_1) -  \ch U(t_0) \, = \, (b+B) \Tch (U, \dot{U}) .
 \end{align}

 With this notation established, we now proceed to define the relative realization of the
 Chern character, to be denoted $\chr\colon K_\bullet(\cA/\cJ) \to HC^{per}_{\bullet}(\cA\colon\cJ)$.
The way in which the relative Chern character will be defined is patterned on the standard
construction of the connecting map in the $K$-theory long exact sequence.

Starting with the odd case, let $u \in GL_N(\cA/\cJ)$ be representing a class in
$ K_1(\cA/\cJ)$. Then
$v =  \begin{bmatrix} 0 & -u^{-1}  \\ u &0 \end{bmatrix}$ is in the connected component of the identity of  $GL_{2N}(\cA/\cJ)$,
hence can be lifted to an invertible $V  \in GL_{2N}(\cA)$. Indeed, choosing lifts
$D$ and $Q$ in $M_N(\cA)$ such that $\sigma (D) = u$ and
 $\sigma (Q) = u^{-1}$, let
$V= \begin{bmatrix} S_0 & -(1+S_0)Q  \\ D &S_1 \end{bmatrix}$, where
$S_0 = I - QD$ and $S_1=I-DQ$. For the record, its inverse is
$V^{-1}= \begin{bmatrix} S_0 & (1+S_0)Q  \\ -D &S_1 \end{bmatrix}$.
Let  $ e := \begin{bmatrix} I_N & 0 \\ 0  & 0 \end{bmatrix}$ and
\begin{align} \label{EV}
E(V)=VeV^{-1} = \begin{bmatrix} S_0^2 & S_0(1+S_0)Q  \\
S_1D &1-S_1^2 \end{bmatrix}.
\end{align}
With $J =  \begin{bmatrix} 0 & -I  \\ I & 0 \end{bmatrix}$, form the path of idempotents
\begin{align} \label{preEVt}
  E(t):=  \begin{bmatrix} V & 0  \\ 0 & I\end{bmatrix} e^{-tJ}
 \begin{bmatrix} e & 0  \\ 0 & 0 \end{bmatrix}
e^{tJ} \begin{bmatrix} V^{-1} & 0  \\ 0 & I \end{bmatrix}  , \quad t \in  [0, \frac \pi 2] ,
\end{align}
which begins at $\widetilde{E}(V) = \begin{bmatrix} V e V^{-1}& 0  \\ 0 & 0\end{bmatrix}$ and
ends at $  \widetilde{e}= \begin{bmatrix} 0 & 0  \\ 0 & e\end{bmatrix}$.

We define
 \begin{align} \label{chr-odd}
\chr (u) \, := \, \left(\ch E(V) - \ch e , \,  \Tch (E, \dot{E}) \right)  ,
 \end{align}
and note that by equation \eqref{trans-ev} $\chr (u)$ is a cocycle in the cone complex
\[\cone \left[CC^{per}_\bullet(\cA, \cJ) \to CC^{per}_\bullet(\cA)\right].\]

By means of secondary transgression formulas (see e.g. \cite[(1.42)]{LMP2009})
one can prove that the homology class $\chr [u] := [\chr (u)] \in HC^{per}_{odd}(\cA, \cJ) $
is well-defined, i.e. only depends on the class $[u] \in K_1(\cA/\cJ)$.

\begin{thm} \label{ch=chr-odd}
The diagram
\begin{equation*}
\begin{tikzcd}[row sep=2.5em]
& K_1(\cA/\cJ)  \arrow{dl}{\chr}  \arrow{dr}{\ch} \\
HC^{per}_{odd}(\cA\colon\cJ) \arrow{rr}{\cong \, \kappa_\bullet} && HC^{per}_{odd}(\cA/\cJ)
\end{tikzcd}
\end{equation*}
is commutative.
\end{thm}

\begin{lemma}\label{lemmachrodd}
  Let $\cB$ be a unital algebra and $u\in \cB$ an invertible element. Set $v =  \begin{bmatrix} 0 & -u^{-1}  \\ u &0 \end{bmatrix}$ and
with $I= \begin{bmatrix} 1 & 0  \\ 0 & 1 \end{bmatrix}$,  $e= \begin{bmatrix} 1 & 0  \\ 0 & 0 \end{bmatrix}$, $J= \begin{bmatrix} 0 & -I  \\ I & 0 \end{bmatrix}$
  \begin{equation} \label{preeVt}
  e(t):=  \begin{bmatrix} v & 0  \\ 0 & I\end{bmatrix} \exp(-tJ)
 \begin{bmatrix} e & 0  \\ 0 & 0 \end{bmatrix}
\exp(tJ) \begin{bmatrix} v^{-1} & 0  \\ 0 & I \end{bmatrix}    , \quad t \in  [0, \frac \pi 2] ,
\end{equation}
Then $\Tch (e, \dot{e})  \in   CC^{per}_\bullet(\cB)$ is a cycle and
\begin{equation}\label{lemmaeq}
[\Tch (e, \dot{e})]=[\ch u] \in HC^{per}_{odd}(\cB).
\end{equation}
\end{lemma}

\begin{proof}
Since $\ch e(0) = \ch e(\pi/2)$,
 \begin{equation*}
(b+B)\Tch (e, \dot{e})  =
 \ch e(\pi/2) - \ch e(0) = 0.
\end{equation*}
so $\Tch (e, \dot{e}) $ is a cycle.
Now, a straightforward calculation using equations
\begin{equation*}
e(t)= \begin{bmatrix} vev^{-1}\cos^2 t&\, ve \sin t\cos t\\
 ev^{-1} \sin t\cos t&e \sin^2 t\end{bmatrix} , \quad t \in  [0, \frac \pi 2] .
\end{equation*}
and
\begin{equation*}
  (2e(t) -1)\dot e(t)
=\begin{bmatrix} 0 & ve\\- ev^{-1} &0\end{bmatrix} .
\end{equation*}
shows that $\Tch (e, \dot{e})  =\frac{1}{2}\left(\ch u -\ch u^{-1}\right)$.
The overall factor $\frac 1 2$ appears here because
\begin{align*}
\int_0^{\frac \pi 2} (\sin t)^{2q+1}  (\cos t)^{2q+1}\, dt \, =  \,  \frac 1 2  B(q+1, q+1)
\, =  \,   \frac 1 2 \frac{(q!)^2}{(2q+1)!} .
\end{align*}
Since $[\ch u^{-1}]= -[\ch u] \in HC^{per}_{odd}(\cB)$, the statement follows.

One can also avoid most of the computations by using the following observation. First of all, by functoriality it is sufficient to prove the statement for the algebra $\cB= \C[u, u^{-1}]$. In this case $HC^{per}_{odd}(\cB)$ is one dimensional and is generated by $[\ch u]$. Hence
\[
[\Tch (e, \dot{e})]=C [\ch u]
\]
for some constant $C$. To determine the constant it suffices to
compute the degree $1$ component of $\Tch (e(t), \dot{e(t)})$:
\begin{multline*}
  \int_{0}^{\pi/2} \tr_\otimes e(t) \otimes  (2e(t) -1)\dot e(t) dt =\\ \int_{0}^{\pi/2} \tr_\otimes  \begin{bmatrix} vev^{-1}\cos^2 t&\, ve \sin t\cos t\\
 ev^{-1} \sin t\cos t&e \sin^2 t\end{bmatrix} \otimes \begin{bmatrix} 0 & ve\\- ev^{-1} &0\end{bmatrix} dt = \\
  \int_{0}^{\pi/2}\tr_\otimes \left( -v e \otimes e v^{-1} +ev^{-1} \otimes ve\right)\sin t \cos t dt = \frac{1}{2}(u^{-1} \otimes u - u \otimes u^{-1}).
\end{multline*}
Let $\Phi \in CC^1(\cB)$ be the cyclic cocycle on $\cB$ given by
\[
\Phi(u^k, u^l) = \begin{cases}
                   l, & \mbox{if } k+l=0 \\
                   0, & \mbox{otherwise}.
                 \end{cases}
\]
Then $\Phi \left( \Tch (e, \dot{e}) \right) = \Phi (\ch u)=1$, hence $C=1$.
\end{proof}

 \begin{proof}[Proof of Theorem \ref{ch=chr-odd}] Let $u$ be an invertible element of $\cA/\cJ$. In order to prove that
 \begin{align*}
 \kappa_\bullet \circ \chr [u] = \ch [u]
 \end{align*}
we will show that the path $E(t)$ constructed above has the property that
\begin{align} \label{prev2-chi}
\sigma_* \left(\Tch (E, \dot{E}) \right) \, =  \,  \ch u.
\end{align}

But $\sigma\left(E(t)\right) = e(t)$, where $e(t)$ is a path of idempotents in $\cA/\cJ$ given by \eqref{preeVt}. Therefore
$\sigma_* \left(\Tch (E, \dot{E})  \right)= \Tch (e(t), \dot{e(t)})$, and the statement follows from
 \eqref{lemmaeq} in Lemma \ref{lemmachrodd}.
 \end{proof}

\begin{remark}\label{algktheory}
 Note that the paths $E(t)$, $e(t)$ are polynomial in $\sin t$, $\cos t$. Therefore all the transgression formulas hold in the algebraic cyclic complex,
without any completions of tensor products. As a consequence, the same proof shows that the following diagram commutes:
  \begin{equation*}
\begin{tikzcd}[row sep=2.5em]
& K^{alg}_1(\cA/\cJ)  \arrow{dl}{\chr}  \arrow{dr}{\ch} \\
HC^{-}_{1}(\cA\colon\cJ) \arrow{rr}{\cong \, \kappa_\bullet} && HC^{-}_{1}(\cA/\cJ)
\end{tikzcd}
\end{equation*}
Here $K^{alg}_1(\cA/\cJ)$ denotes the algebraic $K$-theory, $HC^{-}_{1}(\cA/\cJ)$ is the negative cyclic homology and
$HC^{-}_{\bullet}(\cA\colon\cJ)$ is the homology of the mapping cone $\cone \left[CC^{-}_\bullet(\cA, \cJ) \to CC^{-}_\bullet(\cA)\right]$.
\end{remark}

Passing now to the even case, let $p=p^2 \in M_\infty\left((\cA/\cJ)^+\right)$.
 Choose a lift $P \in M_\infty (\cA)$ of $p$.
 Then  $U(t) = e^{2\pi i t P} $, $t \in [0,1]$
 lifts the loop   $ e^{2\pi i t p} = 1 - p +  e^{2\pi i t }p$,
  $t \in [0,1]$; in particular $ U(1) - U(0) \in M_\infty (\cJ)$.
   Furthermore, one has
   \begin{align} \label{TchU}
 \ch U(1) = \ch U(1) - \ch U(0) &= (b+B) \Tch (U, \dot U) \quad \text{where} \\
  \Tch(U , \dot{U}) &= \int_{0}^{1} \slch(U , \dot{U})(t) dt .
\end{align}
By definition,
 \begin{align} \label{chr-ev}
\chr p \, := \, \frac{1}{2 \pi i}\left(-\ch U(1) , \,  \Tch (U, \dot{U}) \right)  ,
 \end{align}
which due to \eqref{TchU} is a cocycle in
$\cone \left[CC^{per}_\bullet(\cA, \cJ) \to CC^{per}_\bullet(\cA)\right]$.

Again by means of secondary transgression formulas (see \cite[(1.15)]{LMP2009}),
one proves that the homology class
 $\chr [p] := [\chr p] \in HC^{per}_{odd}(\cA, \cJ) $
is well-defined.

\begin{thm} \label{ch=chr-ev}
The diagram
\begin{equation*}
\begin{tikzcd}[row sep=2.5em]
& K_0(\cA/\cJ)  \arrow{dl}{\chr}  \arrow{dr}{\ch} \\
HC^{per}_{ev}(\cA\colon\cJ) \arrow{rr}{\cong \, \kappa_\bullet} && HC^{per}_{ev}(\cA/\cJ)
\end{tikzcd}
\end{equation*}
is commutative.
\end{thm}

 \begin{proof}
The proof is similar to that of Theorem \ref{ch=chr-odd}.
We only need to verify the identity
\begin{equation}\label{prodd2-chi}
 \sigma_*\left(\Tch (U, \dot U)\right) =2\pi i \ch p ,
\end{equation}
which in turn follows immediately from the next lemma.

\begin{lemma}
Let $\cB$ be  an algebra and $p \in \cB$ an idempotent. Form a loop of invertibles $u(t):= 1-p+e^{2 \pi it} p$, $0 \le t \le 1$. Then
\begin{align} \label{reduced2}
\Tch(u, \dot{u}) = 2 \pi i \ch p
\end{align}
\end{lemma}

It is not difficult to verify   \eqref{reduced2} by direct computation. More elegantly,
 one can reduce to the case when $\cB$ is a unital algebra generated by an idempotent $p$,
and observe that the chain $\Tch (u , \dot{u})$
satisfies the cycle identity
\[
(b+B) \sigma_*(\Tch (u, \dot u) ) = \ch(\sigma_* (u(1) - u(0))) \, = \, 0 .
\]
Since its degree $0$ component is $2 \pi i p$
by \cite[Proposition 1.1]{GS1989} it must coincide with $2 \pi i\ch p$.
\end{proof}

\subsection{Higher indices of elliptic operators -- the even case}

The higher analytic indices of a $\Z_2$-graded elliptic operator $D$ have been introduced
in~\cite[\S2]{CM1990}. For any
$\phi \in C^{2q} (M)$ the corresponding index  ${\rm Ind}_{\phi} (D)$
was defined as follows:
\begin{align} \label{hi-old}
 {\rm Ind}_{\phi} (D) \, =  \, \Tr_\phi \left(E(V), \ldots , E(V) \right) ,
 \end{align}
where the idempotent $E(V)$ is constructed as in \eqref{EV}.
The right hand side makes sense whenever the support
of $E(V)^{\otimes (2q+1)} $ is
contained in the support of the locally zero boundary $\delta \phi$, which
can always be arranged
by choosing a parametrix  $Q$ of $D$ with support sufficiently close to the
diagonal.

We shall give below an enhanced version of the definition, which in particular
removes the support restrictions. It relies on the relative Chern character
introduced in \S \ref{RelCh}, applied to the following two cases:
\begin{itemize}
\item[(I)] $\cA = \Psi (M)$ is the algebra of classical pseudodifferential operators,
$\cJ = \Psi^{-\infty}(M)$ is the ideal of smoothing operators,
and the quotient algebra $\cS(M)$ is the algebra of complete symbols;
\item[(II)]
$\cA = \Psi^0 (M)$ is the algebra of bounded pseudodifferential operators,
 $\cJ = \cJ^0(M)$ is the ideal of trace class pseudodifferential operators, and
 $\cS^0(M)$ is the corresponding algebra of symbols.
\end{itemize}

A symbol $u \in M_N(\cS(M))$ is called elliptic if it is a complete symbol of an elliptic pseudodifferential operator, i.e. it has an inverse $u^{-1}  \in M_N(\cS(M))$ and
$\ord u +\ord u^{-1}=0$.

Recall that constructions of the Section \ref{RelCh} allow to associate with an elliptic symbol $u$ the following pseudodifferential idempotents.

First  choose $D$ and $Q$ in $M_N(\Psi(M))$ such that $\sigma (D) = u$ and
 $\sigma (Q) = u^{-1}$. Then set
$V= \begin{bmatrix} S_0 & -(1+S_0)Q  \\ D &S_1 \end{bmatrix}$, where
$S_0 = I - QD$ and $S_1=I-DQ$.
Let  $ e := \begin{bmatrix} I_N & 0 \\ 0  & 0 \end{bmatrix}$ and
\begin{align}
E(V)=VeV^{-1} = \begin{bmatrix} S_0^2 & S_0(1+S_0)Q  \\
S_1D &1-S_1^2 \end{bmatrix}.
\end{align}
One also has the following path from $\begin{bmatrix} E(V) & 0  \\
0 & 0 \end{bmatrix}$ to $\begin{bmatrix} 0 & 0 \\
0 &e \end{bmatrix}$:
\begin{equation*}
E(t)= \begin{bmatrix} VeV^{-1}\cos^2 t&\, Ve \sin t\cos t\\
 eV^{-1} \sin t\cos t&e \sin^2 t\end{bmatrix} , \quad t \in  [0, \frac \pi 2] .
\end{equation*}

\begin{defn} \label{def-hiu}
The {\em higher analytic index} of an elliptic symbol $u \in GL_N(\cS(M))$, associated to
an Alexander--Spanier cocycle $\phi \in C^{2q}_\wedge (M)$, is defined by the formula
 \begin{align} \label{hii}
\begin{split}
 \ind_{\phi} (u) &:= \htr_\phi \left( \ch E(V) - \ch e \right) +
 \htr_{\delta \phi} \left(\Tch (E , \dot E) \right) \\
& \equiv \Tr_\phi \left( \ch E(V) - \ch e \right) +
 \Tr_{\delta \phi} \left(\Tch (E , \dot E) \right).
\end{split}
 \end{align}
 \end{defn}
 The second identity, suggesting that both summands of the formula are independent of
 the regularization of the trace, requires an explanation. Firstly,
 since  ${\delta \phi}$ vanishes
in a neighborhood of the iterated  diagonal $\Delta_{2q+1}M$,
$\htr$ is applied to a trace class
operator, and so the notation $\Tr_{\delta \phi}$ is justified. Secondly, we
note that $\sigma(E(V)) = \sigma (e)$.
This implies that, in the second equality of the formula \eqref{hii},
$\htr$ is applied to an operator with vanishing
symbol, hence to a trace class operator.

The higher index $\ind_{\phi} (u)$ depends only on
the $K$-theory class $[u] \in K_1^{alg}(\cS(M))$ and on
the cohomology class $[\phi] \in H^{2k}_{AS}(M)$. This can be gleaned directly
from the very definition, but it also follows from
the next statement which computes this index directly from the Chern character of the
symbol, generalizing the equation \eqref{chi-map-1}.

 \begin{thm} \label{prop-hii}
  Let $u \in GL_N(\cS(M))$ be an elliptic symbol and let
  $\phi \in C^{2q}_\wedge (M)$ be an Alexander--Spanier cocycle. Then
  \begin{equation}  \label{even-chi}
\ind_{\phi}(u) \, = \, \chi_\phi (\ch u).
\end{equation}
 \end{thm}

 \begin{proof} From the definition \eqref{chi-map} of the map $\chi_\phi$,
 the transgression identity \eqref{trans-ev} and equation \eqref{hii} it follows that

\begin{align}   \label{prev-chi}
 \ind_{\phi} (u) \, =  \, \chi_\phi \left(\sigma_*(\Tch (E , \dot E)) \right) ,
 \end{align}
so the proof is achieved by invoking the identity \eqref{prev2-chi}, see also
Remark \ref{algktheory}.
 \end{proof}
When computed on symbols of the form
$u = \sigma (D)$, where $D$ is an
elliptic pseudodifferential operator on $M$, the higher index $\ind_{\phi}(u)$
coincides, up to a normalizing factor, with  the original definition \eqref{hi-old}.
Indeed, if both $D$ and its parametrix $Q$ can be chosen to be supported sufficiently close
to the the diagonal. In this case $\Tch (E , \dot E)$ is supported close to the diagonal as well and hence.
$\Tr_{\delta \phi} \left(\Tch (E , \dot E) \right) = 0$. Therefore
\eqref{hii} reduces to
\begin{align*}
 \ind_{\phi} (\sigma(D))  =
 (-1)^q \frac{(2q)!}{q!} \Ind_\phi (D) .
 \end{align*}

We now proceed to establish some properties of the higher index pairing
which will allow to show the cases (I) and (II) mentioned above lead to the same outcome.

\begin{prop}\label{additive}
Let $u$, $v$ be two elliptic symbols and $\phi$ an Alexander--Spanier cocycle. Then
\[
\ind_\phi uv = \ind_\phi u +\ind_\phi v
\]
\end{prop}
\begin{proof}
By Theorem \ref{prop-hii}, $\ind_{\phi}(uv)= \chi_\phi (\ch (uv))= \chi_\phi(\ch u+ \ch v) = \chi_\phi(\ch u)+\chi_\phi(\ch v) = \ind_{\phi}(u)+\ind_{\phi}(v)$.
\end{proof}

\begin{lemma}  \label{ind=0}
Let $\Delta\in \Psi(M)$ be a positive, invertible, elliptic operator,
and let $\psi \in C_{\wedge, 0}^{2q+1}(M)$ such that $\delta\psi =0$.
Then
\begin{align*}
 \Tr_\psi (\Delta^{-1}, \Delta, \ldots, \Delta^{-1}, \Delta)=0.
\end{align*}
 \end{lemma}

\begin{proof}
Using the identities of Lemma \ref{cyclic}, extended to pseudodifferential operators in
$\Psi_{\log}(M)$, it is easily seen that for any locally zero cochain
 $\psi \in C_{\wedge, 0}^{2q+1}(M)$ one has
\begin{multline*}
\frac{d}{dz} \Tr_\psi (\Delta^{-z}, \Delta^z, \ldots, \Delta^{-z}, \Delta^z) =\\
 -(q+1) \Tr_{\delta \psi} (\log \Delta, \Delta^{-z}, \Delta^z, \ldots, \Delta^{-z}, \Delta^z)
 =0, \ z \in \C .
\end{multline*}
When $\delta\psi =0$ it follows that
$\Tr_\psi (\Delta^{-z}, \Delta^z, \ldots, \Delta^{-z}, \Delta^z)$ is independent of $z$,
and the proof is achieved by equating its values at $z=1$ and $z=0$.
\end{proof}

\begin{prop}\label{acyclic}
Let $\Delta\in \Psi(M)$ be a positive, invertible, elliptic operator with symbol
 $v=\sigma(\Delta)$, and let $\phi \in C_{\wedge}^{2q}(M)$ be  a cocycle.
Then
\[
\ind_\phi (v) =0.
\]
\end{prop}
\begin{proof}
Being invertible in $\Psi(M)$, $\Delta$ defines
a class $[\Delta] \in K_1^{alg}(\Psi(M))$. Its
Chern character in negative cyclic cohomology
$\ch \Delta \in CC_\bullet^{-} (\Psi(M))$ (see \cite{Loday} for the relevant definitions) is a
lift of $\ch v$. Using Theorem \ref{prop-hii} and Lemma \ref{ind=0}, it follows that
\begin{multline*}
 \ind_{\phi} (v) = \chi_\phi (\ch \Delta) = -\Tr_{\delta \phi} (\ch \Delta)=  \\
 -(-1)^qq!\Tr_{\delta \phi}(\Delta^{-1}, \Delta, \Delta^{-1}, \ldots , \Delta)=0 .
\end{multline*}
\end{proof}

Combining the results of Propositions \ref{additive} and \ref{acyclic}
we obtain the following corollary, which reduces the calculation of higher indices
to the case of symbols of order $0$.

\begin{cor} \label{reduce}
Let  $\Delta$ be the Laplace operator associated with a Riemannian metric on $M$.
For any elliptic symbol $u$ of order $m$, and any cocycle  $\phi \in C_{\wedge}^{2q}(M)$
one has
\[
\ind_\phi(u) = \ind_\phi \left(u\, \sigma(\Delta)^{-\frac m2}\right) .
\]
\end{cor}

We now turn our attention to the case (II), of operators of order $0$ and the
correspoding quotient algebra of symbols $\cS^0(M)$.
Any choice of a quantization map establishes a vector space isomorphism
between $\cS^0(M)$ and the
direct sum of finitely many copies of $C^\infty(S^*M)$, which can be used to endow
$\cS^0(M)$ with a structure of Fr\'echet space. While the isomorphism itself
depends on the choice of quantization, the induced topology does not.
 With this topology the set of invertible elements in $\cS^0(M)$ is open and the
 operation of inversion is continuous. The set of elliptic symbols (for a rank $N$ trivial bundle)
 is all of $GL_N(\cS^0(M))$.

An elliptic symbol $u \in  GL_N(\cS^0(M))$ defines an element
of the topological $K$-group $K_1(\cS^0(M))$. It follows from Theorem \ref{prop-hii},
that the higher analytic index
  $\ind_\phi(u)$ depends only on the class $[u] \in K_1(\cS^0(M))$ and
  the class $[\phi] \in H_{AS}^{ev} (M)$.

Now the canonical projection $\sigma_{pr} \colon  \cS^0(M) \rightarrow C^\infty(S^*M) $
 onto the principal part of the symbol induces an isomorphism in topological $K$-theory,
\begin{align*}
   K_1(\cS^0(M)) \cong K_1(C^\infty(S^*M)) = K^1(S^*M).
 \end{align*}
Via this isomorphism, the higher index pairing turns into a pairing of $K^1 (S^*M)$.
with $H_{AS}^{ev} (M)$.

 Next, let $\pi^*(K^1 (M))$ be the subgroup of $K^1(S^*M)$ made
of classes $[u]$ lifted from $M$. Choosing the tautological
lift $V$ of
 $  v= \begin{bmatrix} 0 & -u^{-1}  \\ u &0 \end{bmatrix}$ to $GL_{2N}\left( \Psi^0(M)\right)$,
one has $V eV^{-1} = e$, so $\ch (E(V)) - \ch e =0$. On the other hand,  since Schwartz kernel of components of
$\Tch (E , \dot E)$  are distributions supported on the diagonal and $\delta \phi$ is locally $0$, $\ind_{\phi} (u) =0$.

We sum up the conclusion in the following remark.

\begin{rem} \label{descent1}
The higher index map associated to any $ [\phi]\in H^{ev}_{AS}(M, \C)$
descends to the quotient, $\ind_{\phi} \colon K^1(S^*M)/\pi^*(K^1 (M)) \to \C$,
and therefore is completely determined by its values on the principal symbols of elliptic
operators.
\end{rem}

To illustrate the computation of the higher index by means of
non-localizable parametrices, we relate it to the heat operator.This, of course,
applies to elliptic operators of order higher than $0$.

As in \cite[\S 2]{CM1990}, given an elliptic operator $D \in \Psi^1(M)$ one can form
out of the heat operator the parametrix
 \begin{align*}
Q(D) \, = \, \frac{I - e^{-\frac 12 D^*D}}{D^*D} D^* \, \equiv \, \int_0^1 e^{-\frac s2 DD^*} D^* \, ds .
\end{align*}
Then  $\displaystyle S_0 (D) \, = \, e^{-\frac 12 D^*D}$,
 $\displaystyle S_1 (D)  \, = \, e^{-\frac 12 DD^*} $,  and so
 \begin{align*}
V (D) = \begin{bmatrix} D & e^{-\frac 12 DD^*} \\ \\
 -e^{-\frac 12 D^*D} &  \frac{I - e^{-D^*D}}{D^*D} D^* \end{bmatrix},
V (D)^{-1} = \begin{bmatrix}  \frac{I - e^{-D^*D}}{D^*D} D^* & - e^{-\frac 12 D^*D} \\ \\
 e^{-\frac 12 DD^*}  & D \end{bmatrix} ;
\end{align*}
 the resulting idempotent $E(D) = V (D)\, e \, V(D)^{-1} $ is
 \begin{align*}
E(D)\, = \,
 \begin{bmatrix} I - e^{-DD^*} & - D e^{-\frac 12 D^*D} \\ \\
 - \frac{I - e^{-D^*D}}{D^*D} e^{-\frac 12 D^*D}D^*  & e^{- D^*D} \end{bmatrix}  .
\end{align*}

 \begin{rem} \label{McS}
Replacing $D$ by $tD$
 one obtains the following extension of the McKean-Singer formula:
\begin{align*}
 \ind_{\phi} (u) \, = \Tr_\phi \left(\ch E(V(tD)) -\ch e  \right) +
 \Tr_{\delta \phi} \left(\Tch (E(tD) , W(tD)) \right) ,
\end{align*}
 for any $t > 0$.
\end{rem}
In particular letting $t \to \infty$ yields the equality
\begin{align} \label{hii-infty}
 \ind_{\phi} (u) \, = \Tr_\phi \left(\ch E_\infty - \ch e\right) +
 \Tr_{\delta \phi} \left(\Tch (E_\infty , W_\infty) \right) .
\end{align}
The limit operators are obtained by replacing the
heat kernel parametrix $Q(D)$ with the
partial inverse $Q_\infty =  D^{-1}(1-H_1)= \left(D (1-H_0)\right)^{-1}$;
thus $V_\infty = \begin{bmatrix} D & H_1\\
-H_0 & Q_\infty \end{bmatrix} $, $ V_\infty^{-1} = \begin{bmatrix} Q_\infty  & -H_0\\
 H_1 & D\end{bmatrix}$ and $E_\infty =  \begin{bmatrix} I-H_1&  0 \\ 0 & H_0 \end{bmatrix}$,
 where $H_0$, resp. $H_1$ are the projections onto $\Ker D$, resp. $\Ker D^*$.

 The first term of the index formula \eqref{hii-infty} is easy to compute.
If $\phi =  \sum_i f^i_0 \otimes f^i_1\otimes \ldots \otimes f^i_{2q} $ then
\begin{align*}
\begin{split}
\Tr_\phi (\ch E_\infty &- \ch e)
 = (-1)^q \frac{(2q)!}{q!}
  \sum_i \left(\Tr \left(H_0 f^i_0 H_0 f^i_1  \ldots H_0 f^i_{2q}\right) \right.\\
 & \qquad  \qquad - \Tr \left.\left(H_1f^i_0 H_1 f^i_1 \ldots H_1f^i_{2q}\right)\right).
\end{split}
\end{align*}
An explicit expression for the second term is also computable but  more cumbersome.
However, if the operator $D$ is invertible the index formula
reduces to the second term, and has an attractive form.

\begin{prop} \label{prop-hii2}
Let $D \in M_N(\Psi(M))$ be an invertible elliptic operator and
let $u = \sigma (D)$. Denoting $S_f:= D^{-1} [ D, f] = f - D^{-1} f D$, one has
\begin{align}  \label{hii2}
\ind_{\phi} (u) = (-1)^{q} q!  \Tr \left(\sum_i
 S_{f^i_0} S_{f^i_1}  \ldots S_{f^i_{2q}}\right) .
\end{align}
 In terms of the phase operator $F= D|D|^{-1}$, one has
\begin{align}  
\ind_{\phi}[u] = (-1)^{q} q! \Tr \left( \sum_i
F [F, f^i_0] \ldots [F, f^i_{2q}]\right).
\end{align}
\end{prop}

 \begin{proof} Since $D$ is invertible and elliptic,
 its inverse $D^{-1}$ is also pseudodifferential. The
Chern character $\ch D \in CC^\bullet (\Psi^0(M))$ is an obvious
lift of $\ch u$. Using Proposition \ref{prop-hii}, it follows that
 \begin{align*}
 \ind_{\phi} (u) &= \chi_\phi (\ch u) = \Tr_{\delta \phi} (\ch D) \\
&= (-1)^qq!\Tr_{\delta \phi}(D^{-1}, D, D^{-1}, \ldots , D) .
\end{align*}
Due to the antisymmetry of $\phi$ and the specific form of the above
expression, it is easily seen that only the extreme terms of $\delta \phi$
contribute to the formula, and therefore
 \begin{align*}
 \ind_{\phi} (u)
&= (-1)^{q} q! \Tr \sum_i \left(f^i_0 D^{-1}f^i_1 \ldots  Df^i_{2q} -D^{-1} f^i_0Df^i_1\ldots f^i_{2q}D\right)\\
 &= (-1)^{q} q! \Tr \left( \sum_i
 S_{f^i_0} S_{f^i_1}  \ldots S_{f^i_{2q}}\right).
\end{align*}
 \end{proof}

\subsection{Higher indices of elliptic operators -- the odd case}

Thanks to the relative Chern character mediation, the definition of the higher indices
in the odd case is completely parallel to the one given in the even case.

\begin{defn} \label{def-hip}
Let  $p^2 = p \in M_N(\cS^0(M))$, and let
$\phi \in C_\wedge^{2q-1}(M)$ be an Alexander--Spanier cocycle. Let $U(t)$ be a lift of the loop $u(t)=1-p+e^{2 \pi i t} p$, $t  \in [0, 1]$, to a loop in
$M_N(\Psi^0(M))$.
The corresponding \textit{higher analytic index} of the idempotent symbol
is defined by the formula
\begin{align} \label{hiidem}
 \ind_{\phi} (p) \, = \frac{1}{2 \pi i}\left(\Tr_\phi \left(\ch U(0) \right) - \Tr_\phi \left(\ch U(1) \right)+
\Tr_{\delta \phi} \left(\Tch (U , \dot{U}) \right)\right) .
\end{align}
  \end{defn}
As in the even case, the notation $\Tr_\phi$ and  $\Tr_{\delta \phi} $
makes sense because $\sigma (U(1)) = \sigma (U(0)) =0$ on the one hand,
and $\delta \phi$ is locally zero on the other hand.

The odd analogue of Theorem \ref{prop-hii} computes the higher index
directly from the symbol.

\begin{thm} \label{prop-hio}
With the above notation one has
  \begin{align}  \label{odd-chi}
\ind_{\phi}(p) \, = \,   \chi_\phi (\ch p) , \qquad \forall \, [p] \in K_0(\cS^0(M)) .
\end{align}
\end{thm}

\begin{proof}
The claimed identity is obtained by
combining \eqref{prodd2-chi} with the transgression formula \eqref{TchU}.
\end{proof}

Again, this implies that $\ind_{\phi} (p)$  only depends on the $K$-theory
class $[p] \in K_0(\cS^0(M))$ and the cohomology class
$[\phi] \in H^{odd}_{AS}(M)$, and via the
canonical isomorphism  $K_0(\cS^0(M)) \cong K_0(C^\infty(S^*M))$
the higher indices are defined on
$K^0 (S^*M)= K_0(C^\infty(S^*M))$. If $p^2 = p \in M_N(C^\infty(M))$,
the loop $ e^{2\pi i t p} $ admits a tautological lift and so $\ch U(1)=\ch (\Id) =0$. On
the other hand $\Tr_{\delta \phi} \left(\Tch (U , \dot{U}) \right) =0$, because
 the components of $\Tch (U , \dot{U})$ are all supported on the diagonal.
So we arrive at the similar conclusion as in the even case.

\begin{rem} \label{descent0}
The higher index map associated to any $ [\phi]\in H^{odd}_{AS}(M, \C)$
descends to the quotient, $\ind_{\phi} \colon K^0(S^*M)/\pi^*(K^1 (M)) \to \C$,
and therefore is completely determined by its values on the principal symbols of elliptic
operators.
\end{rem}

In turn, those higher indices are explicitly expressed by the following
 ``higher Toeplitz index'' counterpart of the ``higher analytic index'' formula \eqref{hii2}.

  \begin{prop} \label{prop-hiidem2}
  Let $D$ be a   selfadjoint elliptic pseudodifferential operator,
 $P$ the projection onto its positive spectrum, and let $p= \sigma(P)$.
If $\phi =  \sum_i
f^i_0 \otimes f^i_1\otimes \ldots \otimes f^i_{2q-1} \in C_\wedge^{2q-1}(M)$
is an Alexander--Spanier cocycle, then
\begin{align}  \label{hiidem2}
 \ind_{\phi} (p) =   (-1)^{q}\frac{(2q)!}{ q!} ! \Tr \left(\sum_i
T_{f^i_0} T_{f^i_1}  \ldots T_{f^i_{2q-1}}\right) ,
\end{align}
where $T_f:= P f P$.
\end{prop}

\begin{proof}
Since $P$ is a projection, $\ch U(1)= \ch (\Id) =0$, and so
\begin{align*}
2 \pi i \cdot \ind_{\phi} (p) \, = \,   \Tr_{\delta \phi} \big( \Tch_{2q}(U , \dot{U}) \big) .
\end{align*}
For later use, let us record that
\begin{align} \label{odd-cocy}
(b+B)\left(\Tch (U , \dot{U}) \right)\, = \,  0  ,
\end{align}
which makes $\Tch (U , \dot{U}) $ a
cocycle in $CC^{per}_\bullet (\Psi^0(M))$.

Apart from the overall numerical factor, each component of the cocycle
$\Tch_{2q}(U , \dot{U})$ has all the factors equal to $\Id - P$, except
one which is $U^{-1}\dot{U} = 2\pi i P$. After obvious cancelations due
to antisymmetrization, a closer look reveals that
the contribution to the pairing of the last $2q$ terms in the expression of
$\delta \phi$ is nil, because two of the summands vanish and the rest of them
successively cancel each other.
The only contribution comes from the the first term of $\delta \phi$, i.e. from
 $ \sum_i 1 \otimes f^i_0 \otimes f^i_1 \otimes \ldots \otimes f^i_{2q-1}$
and it yields the claimed formula.
\end{proof}

\subsection{Suspended indices (the odd case)}
In~\cite[\S3]{MW1993} (cf. also \cite[\S4]{MW1994}) the pairing of selfadjoint elliptic operators
with odd-dimensional Alexander--Spanier
cohomology has been handled via suspension, by means of a suspended Chern character.
We proceed to prove that the corresponding higher indices
coincide with those we have defined above.

Let $D$ be a self-adjoint elliptic operator. By adding a scalar multiple of the
identity, one can assume that $D$ is invertible.
Let $F \in \Psi^0(M)$ be a selfadjoint operator such that
$F - \frac{D}{|D|} \in \Psi^{-\infty}(M)$, and whose
Schwarz kernel of is supported in a sufficiently small neighborhood of
the diagonal, to be specified later. (For the construction of such operators
$F$, see \cite[Lemma 4.1]{MW1994}).

Consider then the loop of elliptic operators:
\begin{align} \label{loopdow}
\begin{split}
&\Phi(\theta)= \left\{\begin{matrix}  \cos \theta \cdot I + i \sin \theta \cdot F ,
& \qquad 0 \leq \theta \leq \pi ,\\ &&\\
( \cos \theta  + i \sin \theta)\cdot I , & \qquad \pi \leq \theta \leq 2\pi .
\end{matrix}\right.
\end{split}
\end{align}
Then
\begin{align} 
\begin{split}
&\Psi(\theta)= \Phi(\theta)^*= \left\{\begin{matrix}  \cos \theta \cdot I - i \sin \theta \cdot F ,
& \qquad 0 \leq \theta \leq \pi ,\\ &&\\
( \cos \theta  - i \sin \theta)\cdot I , & \qquad \pi \leq \theta \leq 2\pi .
\end{matrix}\right.
\end{split}
\end{align}
is a loop of parametrices of $\Phi(\theta)$.
Out of them one forms the loop of idempotents
\begin{equation*}
E_\theta=V_\theta \begin{bmatrix} 1 & 0\\ 0 &0 \end{bmatrix}V_\theta ^{-1}
=\begin{bmatrix} S(\theta)^2 & S(\theta)(1+S(\theta))\Psi(\theta)   \\
S(\theta)\Phi(\theta) &1-S(\theta)^2 \end{bmatrix}, \ 0 \le \theta \le 2 \pi.
\end{equation*}
where $S(\theta)= 1-\Phi(\theta)\Psi(\theta)= 1- \Psi(\theta) \Phi(\theta)$ and
\begin{equation*}
V_\theta  =\begin{bmatrix} S(\theta) & -(1+S(\theta))\Psi(\theta)  \\ \Phi(\theta) &S(\theta) \end{bmatrix}
\end{equation*}
The loop $E= \{E_\theta ; 0 \le \theta \le 2 \pi \}$ represents an element in
$K_0(S \Psi^0 (M))$, and its
suspended Chern character, as defined in~\cite[\S 1.5]{MW1993}, is given by
the cycle
\begin{align} \label{Sch}
\Sch (E):=  \Tch (E, E') =  \int_{0}^{2\pi}  \slch \left(E _\theta, \frac{dE}{d\theta}\right)  d\theta;
\end{align}
$(b+B) \Sch (E) = 0$, since $E_0=E_{2 \pi} = \begin{bmatrix} 1 & 0\\ 0 &0 \end{bmatrix}$,
so $\Sch (E)$ is indeed a cycle
in the cyclic homology bicomplex of  $\Psi^0 (M)$.

If now $\phi\in C_\wedge^{2q+1}(M)$ is an
Alexander--Spanier cocycle, we define the corresponding
{\em suspended index} by the formula
\begin{equation} \label{Sind}
\sind_\phi(D):=\frac{1}{2 \pi i} \Tr_{\phi} \left(\Sch (E)\right) ;
\end{equation}
since $S(\theta)$ are smoothing operators, the the right hand side is well-defined.
 For a fixed cocyle $\phi$ the definition does not depend on the choice of $F$ as above,
 as long as its support is sufficiently localized around the diagonal.

\begin{thm} \label{coincide}
Let $D$ be a selfadjoint elliptic operator and
let $\phi \in C_\wedge^{2q+1}(M)$ be an Alexander--Spanier cocycle. Then
\begin{equation*}
\sind_\phi(D) = \ind_\phi(p) ,
\end{equation*}
where $p$ is the symbol of the orthogonal projection onto the positive spectrum of $D$.
\end{thm}

\begin{proof}
The loop $E \oplus 0$ can be retracted to a constant loop
via a family of loops $[0,1] \ni t \mapsto \tilde{E}(t)$, with $\tilde{E}(0) = E$
and $\tilde{E}(1) = e$, where
\begin{equation*}
\tilde{E}_\theta(t):= \begin{bmatrix} V_\theta &0 \\ 0 &1 \end{bmatrix}\exp\left(\frac{\pi}{2} tJ\right)
 \begin{bmatrix} e &0 \\ 0 &0 \end{bmatrix} \exp\left(-\frac{\pi}{2} tJ\right)
  \begin{bmatrix} V_\theta &0 \\ 0 &1 \end{bmatrix}^{-1},
\end{equation*}
and as before,
\begin{equation*}
e = \begin{bmatrix} 1 &0 \\ 0 &0 \end{bmatrix}, \ J = \begin{bmatrix} 0 & -1 \\ 1 &0 \end{bmatrix}.
\end{equation*}
In order to show that the suspended Chern character is well-defined on $K$-theory,
Moscovici and Wu gave a secondary transgression formula between two homotopic
loops. Specialized to our situation, \cite[Proposition 1.15]{MW1993}  provides
an explicit chain $\TSch(\tilde{E}, \dot{\tilde{E}})$ in the cyclic homology bicomplex
of  $\Psi^0 (M)$ satisfying
\begin{equation}  \label{Strans}
- \Sch (E) \equiv \Sch(E_\theta(1) )- \Sch (E_\theta(0)) =
(b+B)\left(\TSch(\tilde{E}, \dot{\tilde{E}})\right).
\end{equation}

Using the explicit expression of  $\TSch(\tilde{E}, \dot{\tilde{E}})$,  one verifies
by a straightforward calculation that
 \begin{equation*}
  \sigma\left(\TSch(\tilde{E}, \dot{\tilde{E}})\right)= 2 \pi i \ch (p) ,
  \end{equation*}
 where $p = (1+\sigma(F))/2$.
  From this and the transgression formula \eqref{Strans}, by
  the very definition \eqref{chi-map} of the map $\chi_\phi$, we obtain
\begin{align*}
2 \pi i  \chi_\phi(\ch p) = \Tr_{\delta\phi} \left(\TSch(\tilde{E}, \dot{\tilde{E}})\right) +
\Tr_\phi \left(\Sch (E)\right).
\end{align*}
By choosing $F$ with Schwartz kernel supported
 in a sufficiently small neighborhood of diagonal, depending on the vanishing locus
 of $\delta\phi$,
 one can ensure that the first term of the above sum vanishes, and therefore
 \begin{align*}
 \chi_\phi(\ch p) = \frac{1}{2 \pi i}
\Tr_\phi \left(\Sch (E)\right) = \sind_\phi(D) .
\end{align*}
In view of Theorem \ref{prop-hio} this completes the proof.
\end{proof}

\section{Cohomological formulas for higher indices} \label{CoFo}
As explained in the remarks \ref{descent1} and \ref{descent0},
the index pairing between the $K$-theory groups of the symbol algebra
$\cs^0(M)$  and Alexander--Spanier cohomology of $M$
is completely determined by the higher analytic indices of elliptic operators,
which in turn depend only on their principal symbols. This allows to reduce
the cohomological computation of the higher index pairing
to finding explicit cohomological expressions
for higher indices of Dirac-type operators. The latter have already been
established, cf.~\cite[\S 3]{CM1990} for the even case, while in the odd case
similar formulas can be easily derived for the suspended version of higher indices
from the results in~\cite[\S 3]{MW1993}. In view of Theorem \ref{coincide} it only
remains to combine the formulas for higher analytic indices obtained in the previous
section with the corresponding cohomological expressions.

In the even case the cohomological formula for the index of an elliptic operator $D$ is
\[
\ind_\phi(u) = \frac{(-1)^q}{(2\pi i)^q} \int_{M} \pi_*(\ch (\sigma_{pr}(D))) \wedge \Td(TM \otimes \C)\wedge \left( \sum_i f_0^idf_1^i \ldots df_{2q}^i \right);
\]
here $u=\sigma (D)$,
$\sigma_{pr}(D) \in C^\infty(S^*M)$ is its principal symbol, $\pi \colon S^*M \to M$ is the canonical projection, and $\Td(TM \otimes \C)$ is the Todd class of the complexified tangent bundle.

From~\cite[Theorem (3.9)]{CM1990} and Proposition \ref{prop-hii2} we obtain:

\begin{thm} \label{genhh-ev}
 Let $D \in M_N(\Psi(M))$ be an invertible elliptic operator. Set $S_f:= D^{-1} [ D, f] = f - D^{-1} f D$.
 If $\phi =  \sum_i f^i_0 \otimes f^i_1\otimes \ldots \otimes f^i_{2q} \in C_\wedge^{2q}(M)$
is an Alexander--Spanier cocycle, then
\begin{multline}\label{anticomm-ev}
  \Tr \left(\sum_i
 S_{f^i_0} S_{f^i_1}  \ldots S_{f^i_{2q}}\right) = \\ \frac{1}{(2\pi i)^q q!} \int_{M}
  \pi_*(\ch (\sigma_{pr}(D)))
 \wedge \Td(TM \otimes \C)\wedge \left( \sum_i f_0^idf_1^i \ldots df_{2q}^i \right).
\end{multline}
\end{thm}

In the odd case the cohomological formula for the index of a selfadjoint elliptic operator $D$ is
\[
\ind_\phi(p) = \frac{(-1)^q}{(2\pi i)^q} \int_{M} \pi_*(\ch (\sigma_{pr}(p))) \wedge \Td(TM \otimes \C)\wedge \left( \sum_i f_0^idf_1^i \ldots df_{2q-1}^i \right);
\]
here $p=\sigma (P)$, where $P$ is the orthogonal projection on the positive spectrum of $D$.

For a Dirac-type operator $D$ the cohomological expression of $\sind_\phi (D)$ can be computed
by means of the trace formula for the inverse Laplace transform~\cite[Proposition 3.6]{MW1994}
and using Getzler's asymptotic calculus, along the lines invoked
in proving ~\cite[Theorem 4.2]{MW1994}. The calculation per se mirrors the one
performed in~\cite[\S3]{MW1993},
 and yields the right hand side of the above formula.
 Combining this with Theorem \ref{coincide} and
Proposition \ref{prop-hiidem2} we obtain:

\begin{thm}\label{genhh-odd}
  Let $D$ be a   selfadjoint elliptic pseudodifferential operator,
 $P$ the projection onto its positive spectrum. Set $T_f:= P f P$.
If $\phi =  \sum_i
f^i_0 \otimes f^i_1\otimes \ldots \otimes f^i_{2q-1} \in C_\wedge^{2q-1}(M)$
is an Alexander--Spanier cocycle, then
\begin{multline}\label{anticomm-odd}
  \Tr \left(\sum_i
T_{f^i_0} T_{f^i_1}  \ldots T_{f^i_{2q-1}}\right) = \\ \frac{q!}{(2\pi i)^q (2q)!} \int_{M}
\pi_*(\ch (\sigma_{pr}(P))) \wedge  \Td(TM \otimes \C)\wedge
\left(  \sum_i  f_0^idf_1^i \ldots df_{2q-1}^i \right).
\end{multline}
\end{thm}

 The identities \eqref{anticomm-ev} and \eqref{anticomm-odd} generalize the Helton--Howe
 trace formula for the top multi-commutator of Toeplitz operators~\cite[Theorem 7.2]{HH1975}.
 The latter corresponds to the case when $M$ is the cosphere bundle $S^*N$  of another
 manifold $N$, and $D$ is the canonical Spin$^c$-Dirac operator.
 However, since the Helton--Howe result
 applies to arbitrary cochains $\phi \in C^{m}(M)$, $m=\dim M$, we need
 to explain how to modify the complex so that all the top-dimensional Alexander--Spanier
 cochains qualify as cocycles.

Recall that the Alexander--Spanier complex
we used in this paper was defined as the quotient
complex \, $C_{AS}^\bullet(M) =\left\{C^\bullet_\wedge (M)/C_{\wedge, 0}^\bullet(M), \delta \right\}$, \, where
$C^\bullet_\wedge(M)$ consists of smooth indecomposable totally antisymmetric cochains, and
 and $C_{\wedge,0}^q(M)$ consists of those $\phi \in C^\bullet_\wedge(M)$ which vanish in a neighborhood
of the iterated diagonal $\Delta_{q+1}M$.
We can modify this definition by taking instead the quotient by
the larger subcomplex
 \[
 \bar{C}_{\wedge,0}^q(M) :=\left\{\phi \in C^q_\wedge(M) \mid \text{the $m$-th jet of
 $\phi$ vanishes at $\Delta_{q+1}M$} \right\}
 \]
where, recall, $m= \dim M$. It is easy to see that  $ \bar{C}_{\wedge,0}^q(M)$ is indeed a subcomplex. Set
\[
\bar{\bf C}_{AS}^\bullet(M) := \left\{C^\bullet_\wedge (M)/\bar{C}_{\wedge, 0}^\bullet(M), \delta \right\}
\]
We note that, due to the total antisymmetry,
\[
\bar{\bf C}_{AS}^q(M) =0 \text{ for } q>m.
\]
It is also easy to see that the canonical morphism to de Rham complex
\[
\lambda \colon \bar{\bf C}_{AS}^\bullet(M) \to \Omega^\bullet (M)
\]
given by
\[
\lambda( \sum_i f_0^i \otimes f_1^i \otimes \ldots \otimes f_q^i) := \sum_i f_0^i df_1^i  \ldots df_q^i
\]
is well defined and surjective.

Due to the total antisymmetry of a chain $\phi \in C^m_\wedge(M)$,  its $(m-1)$-jet
 vanishes at $\Delta_{m+1}M$;  it is easy to see then that
\[
\lambda \colon \bar{\bf C}_{AS}^m(M) \to \Omega^m (M)
\]
is an isomorphism of vector spaces.
The surjectivity of $\lambda$ implies that the cohomology of the complex
$\bar{\bf C}_{AS}^\bullet (M)$ in degree $m$ coincides with $H^m(M)$. Hence the canonical projection
\[
 \mathbf{C}_{AS}^{\bullet}(M) \to \bar{\bf C}_{AS}^{\bullet}(M)
\]
induces an isomorphism in degree $m$ cohomlogy.  Equivalently, every cohomology class of $\bar{\bf C}_{AS}^\bullet (M)$ of degree $m$
has a representative $\phi \in C^m_\wedge (M)$ such that $\delta \phi \in  C^{m+1}_{\wedge, 0} (M)$.

The construction of the map $\chi$ extends without difficulty to the new complex,
provided that the target is the cyclic complex of  $0$-order complete symbols, thus
yielding the map of complexes
\[
\chi \colon \bar{\bf C}^\bullet_{ AS}(M) \rightarrow CC^{\bullet+1}(\cs^0(M)).
\]
Indeed, Lemma \ref{cyclic} remains true when
 $\phi \in \bar{C}_{\wedge,0}^k(M)$ and
 $A_0, \ldots, A_k\in \Psi^0(M)$, because the vanishing of the
$m$-th jet of $\phi$ at $\Delta_{k+1}M$ guarantees that
$\sum_i A_0 f_0^i A_1 f_1^i \ldots A_k [\log R, f_k^i] \in \Psi^{-m-1}(M)$
and therefore, by equation \eqref{Wres},
 $ \Res \sum_i A_0 f_0^i A_1 f_1^i \ldots A_k [\log R, f_k^i] = 0$.
 Thus, the proof of Proposition \ref{post-cyclic} remains unchanged, provided
 that $\chi_\phi$ is applied to $0$-order pseudodifferential operators.

We finally note that the proofs of the identities
\eqref{anticomm-ev} and \eqref{anticomm-odd}
effectively involve only pseudodifferential operators of order  $0$.
 Since all the top degree cohomology classes in the new complex
 $\bar{\bf C}_{AS}^\bullet(M)$ can be represented by cocycles in ${\bf C}_{AS}^\bullet(M)$, the identities \eqref{anticomm-ev} and \eqref{anticomm-odd}
remain true for all the top degree cocycles (i.e. all degree $m$ cochains) in the new complex
$\bar{\bf C}_{AS}^\bullet(M)$. Recalling the notation
\[
\left[ A_1, A_2, \ldots, A_k\right]:= \frac{1}{k!}\sum_{\tau \in S_k} \sgn \tau A_{\tau(1)} A_{\tau(2)} \ldots A_{\tau(k)} ,
\]
we can therefore conclude that the following extensions of the Helton--Howe formula hold true.

 \begin{cor}
Let $M$ be a connected, closed, smooth manifold of even dimension
$m=2q$. Let $D \in M_N(\Psi^0(M))$ be an invertible elliptic operator.
 Then,  for any $f_0$, $f_1$, \ldots, $f_{2q} \in C^\infty(M)$
\begin{align*}
  \Tr \left[
S_{f_0}, S_{f_1},  \ldots, S_{f_{2q}}\right] = \kappa \frac{(2q+1)!}{(2\pi i)^q q!} \int_{M}  f_0df_1 \ldots df_{2q} ,
\end{align*}
where $\kappa$ is the component of $ \pi_*(\ch(\sigma_{pr}(D))) $ in $H^0(M, \mathbb{Q})$, which is canonically identified with $\mathbb{Q}$.

Also, for any $f_0$, $f_1$, \ldots, $f_{2r} \in C^\infty(M)$ with $r > q$,
\begin{align*}
  \Tr \left[
S_{f_0}, S_{f_1},  \ldots, S_{f_{2r}}\right] = 0 .
\end{align*}
\end{cor}

\begin{cor}
Let $M$ be a connected, closed, smooth manifold of odd dimension
$m=2q-1$. Let $D$ be a  selfadjoint elliptic pseudodifferential operator, and
 $P$ the projection onto its positive spectrum.
Then, for any $f_0$, $f_1$, \ldots, $f_{2q-1} \in C^\infty(M)$, one has
\begin{align*}
  \Tr \left[T_{f_0}, T_{f_1},  \ldots, T_{f_{2q-1}}\right] &= \kappa \frac{ q!}{(2\pi i)^q } \int_{M}
f_0df_1 \ldots df_{2q-1} ,
\end{align*}
where   $\kappa$ is the component of $ \pi_*(\ch(\sigma_{pr}(P)))$ in $H^0(M, \mathbb{Q}) \cong \mathbb{Q}$.

Also, for any $f_0$, $f_1$, \ldots, $f_{2r-1} \in C^\infty(M)$ with $r > q$,
\begin{align*}
  \Tr \left[T_{f_0}, T_{f_1},  \ldots, T_{f_{2r-1}}\right]  = 0 .
\end{align*}
\end{cor}


\end{document}